\documentclass{amsart}

\usepackage{amssymb}
\usepackage[foot]{amsaddr}
\usepackage[all]{xy}

\def\Cc{\mathcal{C}}
\def\Cat{\mathrm{Cat}}
\def\Ee{\mathcal{E}}
\def\Bb{\mathcal{B}}
\def\el{\mathrm{el}}
\def\Cov{\mathrm{Cov}}
\def\Adj{\mathrm{Adj}}
\def\lrto{ \leftrightarrows }
\def\rlto{ \rightleftarrows }
\def\Ll{\mathcal{L}}
\def\Rr{\mathcal{R}}
\def\Sh{\mathrm{Sh}}
\def\Dd{\mathcal{D}}
\def\Sing{\mathrm{Sing}}
\def\Dec{\mathrm{Dec}}
\def\FF{\mathfrak{F}}
\def\Uu{\mathcal{U}}
\def\Vv{\mathcal{V}}

\def\Oo{\mathcal{O}}

\def\opd{\textrm{-ops}}
\def\Fop{\FF\opd}
\def\Top{ \mathrm{Top} }
\def\Set{ \mathrm{Set} }
\def\Gpd{ \mathrm{Gpd} }
\def\inc{ \hookrightarrow }
\def\Ob{ \mathrm{Ob} }
\def\Fcyc{ \FF_{cyc} }
\def\Fmod{ \FF_{mod} }
\def\op{\mathrm{op}}
\def\ito{ \rightarrowtail }
\def\colim{ \mathrm{colim} }
\def\Oo{ \mathcal{O} }
\def\Alg{ \mathrm{Alg} }
\def\NN { \mathbb{N} }
\def\Aut{ \mathrm{Aut} }

\theoremstyle{plain}

\newtheorem{theorem}{Theorem}
\newtheorem{dfn}{Definition}[section]
\newtheorem{thm}[dfn]{Theorem}
\newtheorem{lma}[dfn]{Lemma}
\newtheorem{prp}[dfn]{Proposition}
\newtheorem{cor}[dfn]{Corollary}

\theoremstyle{remark}

\newtheorem{rmk}[dfn]{Remark}

\newtheorem{sct}[dfn]{}

\title{Comprehensive factorisation systems}

\author{Clemens Berger}

\address{Universit\'{e} de Nice, Lab. J. A. Dieudonn\'{e}, Parc Valrose, 06108 Nice Cedex, France}
\email{cberger@math.unice.fr}

\author{Ralph M. Kaufmann}

\address{Purdue University Department of Mathematics, West Lafayette, IN 47907U}
\email{rkaufman@math.purdue.edu}


\date{November 23, 2017}


\begin{document}

\begin{abstract}We establish a correspondence between consistent comprehension schemes and complete orthogonal factorisation systems. The comprehensive factorisation of a functor between small categories arises in this way. Similar factorisation systems exist for the categories of topological spaces, simplicial sets, small multicategories and Feynman categories. In each case comprehensive factorisation induces a natural notion of universal covering, leading to a Galois-type definition of fundamental group for based objects of the category.\end{abstract}

\maketitle

\begin{center}\emph{Dedicated to Peter Freyd and Bill Lawvere on the occasion of their 80$^\mathit{th}$ birthdays}\end{center}

\section*{Introduction}\label{Intro}

The main purpose of this text is to promote a somewhat unusual point of view on orthogonal factorisation systems, based on a minor variation of Lawvere's notion of comprehension scheme \cite{Law1}. A \emph{comprehension scheme} $P$ on a category $\Ee$ assigns to each object $A$ of $\Ee$ a category $PA$ with terminal object $\star_{PA}$, and to each morphism $f:A\to B$ of $\Ee$ an adjunction $f_!:PA\lrto PB:f^*$ satisfying certain axioms. These axioms imply the existence of a full subcategory $\Cov_B$ of the slice category $\Ee/B$ consisting of so called \emph{$P$-coverings}, together with an equivalence of categories $\Cov_B\simeq PB$. A morphism $f:A\to B$ is called \emph{$P$-connected} if $f_!(\star_{PA})\cong\star_{PB}$.\vspace{1ex}

Our main result reads as follows (cf. Theorems \ref{main} and \ref{mainbis}):

\begin{theorem}\label{mainintro}Every consistent comprehension scheme $P$ on $\Ee$ induces a complete orthogonal factorisation system on $\Ee$ with left part consisting of $P$-connected morphisms and right part consisting of $P$-coverings.

Conversely, every complete orthogonal factorisation system on $\Ee$ arises in this way from an essentially unique consistent comprehension scheme on $\Ee$.\end{theorem}

A comprehension scheme is \emph{consistent} precisely when $P$-coverings compose and are left cancellable, cf. Proposition \ref{consistent}. A factorisation system $(\Ll,\Rr)$ (cf. Freyd-Kelly \cite{FK}) is \emph{complete} if pullbacks of $\Rr$-morphisms exist along any morphism of $\Ee$. We show that under this correspondence the comprehension scheme $P$ satisfies \emph{Frobenius reciprocity} in the sense of Lawvere \cite{Law1} if and only if $P$-connected morphisms are stable under pullback along $P$-coverings.

Street and Walters' comprehensive factorisation of a functor \cite{StWa} into an initial functor followed by a discrete opfibration arises in this way from the comprehension scheme which assigns to a category $A$ the category $\Set^A$ of set-valued diagrams. Dually, the comprehension scheme assigning to a category $A$ the category $\Set^{A^\op}$ of set-valued presheaves induces the factorisation into a final functor followed by a discrete fibration. In both cases, the axioms of a comprehension scheme amount to the existence of the \emph{category of elements} for any set-valued diagram, resp. presheaf. It was Lawvere's insight \cite{Law1} that the existence of these discrete Grothendieck constructions is encoded by the existence of a certain adjunction.

One of the leading motivations of this text has been the recent construction by the second author and Lucas \cite{KL} of decorated Feynman categories, which play the role of \emph{Feynman categories of elements}. A Feynman category is a special kind of symmetric monoidal category, and there is a comprehension scheme assigning to a Feynman category $\FF$ the category of strong symmetric monoidal set-valued functors on $\FF$. The resulting comprehensive factorisation of a Feynman functor sheds light on Markl's recent definition of non-$\Sigma$-modular operads \cite{Ma}. Through the $2$-equivalence between Feynman categories and small multicategories (also called coloured operads) we obtain a comprehensive factorisation of a multifunctor which directly extends Street and Walters' comprehensive factorisation of a functor.

Another instructive example is the comprehension scheme which assigns to a well-behaved topological space $A$ the category $\Sh_{loc}(A)$ of locally constant set-valued sheaves on $A$. The resulting comprehensive factorisation factors a continuous map into a map with connected homotopy fibres followed by a topological covering. The category of simplicial sets carries a similar comprehension scheme. The induced simplicial coverings are precisely the Kan fibrations with discrete fibres.

The last two examples suggest that categories $\Ee$ admitting a ``discrete'' comprehension scheme $P$ (i.e. such that the value of $P$ at a terminal object $\star_\Ee$ of $\Ee$ is the category of sets) can be investigated from a Galois-theoretical perspective. We undertake first steps in this direction. We define discrete, connected and locally connected objects using the comprehensive factorisation. Moreover, any based object $\alpha:\star_\Ee\to A$ admits a \emph{universal $P$-covering} $\Uu_{\alpha}\to A$, obtained by comprehensive factorisation of $\alpha$. The group of deck transformations of this universal covering is a natural candidate for the fundamental group $\pi_1(A,\alpha)$. We explore this definition in the aforementioned cases and show that a faithful fibre functor $\alpha^*:PA\to\Set$ factors through the category of $\pi_1(A,\alpha)$-sets whenever the comprehension scheme satisfies Frobenius reciprocity. We give a sufficient condition for monadicity of fibre functors, closely related to Grothendieck's axiomatisation of Galois theory \cite{Gr}.\vspace{1ex}

This article is organised as follows:\vspace{1ex}

Section \ref{scheme} establishes the correspondence between consistent comprehension sche-mes and complete orthogonal factorisation systems. We discuss Frobenius reciprocity and define restriction and extension of comprehension schemes.

Section \ref{feynman} investigates the standard comprehension scheme for Feynman categories and small multicategories, with an application to modular operads.

Section \ref{galois} studies Galois-theoretical aspects of categories with discrete comprehension scheme. We define natural $\pi_0$- and $\pi_1$-functors and investigate fibre functors and their monadicity.

\section{Comprehension schemes and factorisation systems}\label{scheme}

By \emph{comprehension scheme} on a category $\Ee$ we mean a pseudo-functor $P:\Ee\to\Adj_*$ assigning to each object $A$ of $\Ee$ a category $PA$ with distinguished terminal object $\star_{PA}$, and to each morphism $f:A\to B$ an adjuntion $f_!:PA\lrto PB:f^*$ such that Lawvere's comprehension functor
\begin{gather*}
\xymatrix@R=0cm{
 c_B:\Ee/B \ar[r] &  PB\\
 (f:A\to B)\ar@{|->}[r] & f_!(\star_{PA})
}\end{gather*}
has a \emph{fully faithful} right adjoint $p_B:PB\to\Ee/B.$

Note that the functoriality of $c_B$ follows from the existence of unique morphisms $f_!(\star_{PA})\to\star_{PA'}$ in $PA'$ for each $A\to A'$ in $\Ee$.

The unit of the $(c_B,p_B)$-adjunction  at $f:A\to B$ yields a factorisation
\begin{equation}\label{unit}\xymatrix{
 A \ar[r]^{\eta_f}\ar[d]_f &  \el_B(c_B(f))\ar[ld]^{p_Bc_B(f)}\\
 B
}\end{equation}and we say that $f:A\to B$ is a \emph{$P$-covering} if $\eta_f$ is invertible. A morphism $f:A\to B$ of $\Ee$ is said to be \emph{$P$-connected} if $f_!(\star_{PA})\cong\star_{PB}$.

The full subcategory of $\Ee/B$ spanned by the $P$-coverings will be denoted $\Cov_B$. In particular, we have an equivalence of categories $\Cov_B\simeq PB$ for each object $B$.  The comprehension scheme is said to be \emph{consistent} if each $P$-covering $f:A\to B$ induces an equivalence of categories $f_!:PA\simeq PB/f_!(\star_{PA})$. \vspace{1ex}

Let us mention here that what we call a comprehension scheme on $\Ee$ is precisely what Jacobs \cite[Example 4.18]{J} calls a \emph{full Lawvere category} over $\Ee$ showing that this notion is a special case of Ehrhard's \emph{$D$-categories} \cite{E} which are renamed \emph{comprehensive categories with unit}. It is noticeable that a certain amount of our results (such as Lemma \ref{elements}) carry over to the more general context of comprehension categories with unit where the existence of left adjoint functors $f_!$ is not required.

\begin{lma}[cf. \cite{Jbook}, Lemma 10.4.9(i)]\label{elements}The existence of a right adjoint $p_B:PB\to\Ee/B$ of $c_B$ amounts to the existence (for each object $X$ of $PB$) of an \emph{object of elements} $\el_B(X)$ over $B$ having the universal property that for each $h:A\to B$ in $\Ee$ there is a bijection between morphisms $\star_{PA}\to h^*(X)$ in $PA$ and liftings in $\Ee$$$\xymatrix{
 &\el_B(X)\ar[d]^{p_B(X)}\\
 A\ar[r]_h\ar@{.>}[ru]&B
}$$which is natural with respect to morphisms $X\to Y$ in $PB$.\end{lma}

\begin{proof}Morphisms $\star_{PA}\to h^*(X)$ are in one-to-one correspondence with morphisms $c_B(h)=h_!(\star_{PA})\to X$ so that the condition above expresses that the latter correspond to morphisms $h\to p_B(X)$ in $\Ee/B$. Naturality in one variable suffices.\end{proof}

\begin{lma}\label{slice}A comprehension scheme $P$ on $\Ee$ is consistent precisely when for each $B$ in $\Ee$ and each $X$ in $PB$, the map $P\el_B(X)\to PB/c_Bp_B(X)\to PB/X$ is an equivalence of categories.\end{lma}
\begin{proof}Since $p_B(X):\el_B(X)\to B$ is a $P$-covering, and up to isomorphism over $B$ any $P$-covering is of this form, consistency amounts to the condition that $p_B(X)_!:P\el_B(X)\simeq PB/c_Bp_B(X)$ is an equivalence of categories. By definition of a comprehension scheme, the counit $\epsilon_X:c_Bp_B(X)\to X$ is always an isomorphism, and hence the second map above is always an equivalence of categories.\end{proof}

\begin{sct}\textbf{Category of elements.}\label{Cat} The special case where $\Ee=\Cat$ is the category of small categories serves as guideline throughout. We have two comprehension schemes here, given respectively by $PA=\Set^A$ and $P'A=\Set^{A^\op}$. In both cases, the right adjoint is restriction and the left adjoint is given by left Kan extension.

The universal property of $\el_B(X)$, as stated in Lemma \ref{elements}, is satisfied by the comma category $\star\downarrow X$ where $\star$ denotes a singleton functor on the terminal category, cf. Street-Walters \cite{StWa}. This comma category is often called the \emph{category of elements} of $X$. It is a special case of the Grothendieck construction of a functor. In the covariant case, the objects of $\el_B(X)$ are pairs $(b\in B,x\in X(b))$ and the morphisms $(b,x)\to(b',x')$ are those $\phi:b\to b'$ in $B$ for which equality $X(\phi)(x)=x'$ holds. In the contravariant case, one has to dualise twice in order to get $p_B(X):\el_B(X)\to B$.

For each category $B$ and each $X\in\Set^B$, the counit $c_B(p_B(X))\to X$ is invertible (i.e. the right adjoint $p_B:\Set^B\to\Cat/B$ is fully faithful) because any functor $X:B\to\Set$ may be identified with the left Kan extension of the singleton functor $\star:\el_B(X)\to\Set$ along the projection $\el_B(X)\to B$, cf. \cite[Proposition 1]{StWa}.

The comprehension schemes $P,P'$ are consistent, cf. Lemma \ref{slice}. It suffices to consider the functor $p_B(X):\el_B(X)\to B$ defined by a diagram $X:B\to\Set$ (resp. presheaf $X:B^\op\to\Set$). It can be checked by hand that the induced functor $\Set^{\el_B(X)}\to\Set^B/X$ (resp. $\Set^{\el_B(X)^\op}\to\Set^{B^\op}/X$) is an equivalence.

$P$-coverings are precisely discrete opfibrations, and $P'$-coverings precisely discrete fibrations. These two classes of functors compose and are left cancellable so that Proposition \ref{consistent} below is an alternative way to extablish consistency.\end{sct}

\begin{sct}\textbf{Powerset comprehension scheme.} Another example, certainly motivating Lawvere \cite{Law1}, is the powerset functor $P:\Set\to\Adj_*$ assigning to a set $A$ its powerset $PA$, considered as a category via its poset structure. The adjunction $f_!:PA\lrto PB:f^*$ is given by direct/inverse image. The comprehension functor $c_B:\Set/B\to PB$ assigns to a mapping $f:A\to B$ its image $f(A)\in PB$ and the right adjoint functor $p_B:PB\to \Set/B$ assigns to a subset its subset-inclusion. The $P$-coverings are precisely the injective mappings. The comprehension scheme is consistent because any injective mapping $f:A\to B$ induces an isomorphism $PA\cong PB/f(A)$. This follows also from Proposition \ref{consistent} below because injective mappings are composable and left cancellable.\end{sct}

\begin{prp}\label{consistent}A comprehension scheme $P$ is consistent if and only if $P$-coverings compose and are left cancellable.\end{prp}

\begin{proof}Since by definition the left adjoints of a comprehension scheme $P$ compose up to isomorphism, for each morphism $g:B\to C$ the following square of functors$$\xymatrix{\Ee/B\ar[d]_{c_B}\ar[r]^{g\circ -}&\Ee/C\ar[d]^{c_C}\\PB\ar[r]_{g_!}&PC}$$pseudocommutes (i.e. commutes up to isomorphism). If $g$ is a $P$-covering, and $P$-coverings compose, then we get by restriction a pseudocommuting square$$\xymatrix{\Cov_B\ar[d]_{\simeq}\ar[r]^{g\circ -}&\Cov_C\ar[d]^{\simeq}\\PB\ar[r]_{g_!}&PC}$$with vertical equivalences. The latter induces a pseudocommuting square
$$\xymatrix{\Cov_B\ar[d]_{\simeq}\ar[r]^{g\circ -}&\Cov_C/g\ar[d]^{\simeq}\\PB\ar[r]_{g_!}&PC/g_!(\star_{PB})}$$in which the upper horizontal functor is an equivalence (even an isomorphism) whenever $P$-coverings are left cancellable. Therefore, if $P$-coverings compose and are left cancellable, then $P$ is consistent.

Conversely, consider composable morphisms $A\xrightarrow{f}B\xrightarrow{g}C$ in $\Ee$. We shall say that $f_!(\star_{PA})$ is $f$-universal, if liftings of $h:D\to B$ to $f:A\to B$ correspond one-to-one to morphisms $h_!(\star_{PD})\to f_!(\star_{PA})$ in $PB$. According to Lemma \ref{elements}, $f_!(\star_{PA})$ is $f$-universal if and only if $f$ is a $P$-covering. We have thus to show that for a consistent comprehension scheme $P$, if $g_!(\star_{PB})$ is $g$-universal then $f_!(\star_{PA})$ is $f$-universal precisely when $(gf)_!(\star_{PA})$ is $gf$-universal.

Note that there is a morphism $(gf)_!(\star_{PA})\to g_!(\star_{PB})$ obtained by applying $g_!$ to the unique morphism $f_!(\star_{PA})\to\star_{PB}$. Since $P$ is consistent, $g_!$ acts fully faithfully on morphisms with terminal codomain so that the former morphism is unique too. In particular, assuming that $g_!(\star_{PB})$ is $g$-universal, for any $h:D\to C$, a morphism $h_!(\star_{PD})\to g_!(\star_{PB})$ with lifting $\tilde{h}:D\to B$ factors through $(gf)_!(\star_{PA})$ precisely when $\tilde{h}(\star_{PD})$ maps to $f_!(\star_{PA})$, and these two data determine each other.

Therefore, if moreover $f_!(\star_{PA})$ is $f$-universal, then the lifting $\tilde{h}:D\to B$ has itself a unique lifting $\tilde{h}':D\to A$, which implies that $(gf)_!(\star_{PA})$ is $gf$-universal. Conversely, any map $\tilde{h}:D\to B$ may be considered as the lifting of $h=g\tilde{h}:D\to C$ associated with the morphism $g_!(\tilde{h}_!(\star_{PD})\to\star_{PB})$. Henceforth, if $(gf)_!(\star_{PA})$ is $gf$-universal, then the liftings of $\tilde{h}$ to $f$ correspond bijectively to liftings of $h$ to $gf$, i.e. to morphisms $\tilde{h}_!(\star_{PD})\to f_!(\star_{PA})$ so that $f_!(\star_{PA})$ is $f$-universal.\end{proof}

\begin{lma}\label{factorisation}Let $P$ be a comprehension scheme on $\Ee$.
\begin{itemize}\item[(a)]Pullbacks of $P$-coverings exist in $\Ee$ and are again $P$-coverings;\item[(b)]If $P$ is consistent then each morphism of $\Ee$ factors as a $P$-connected morphism followed by a $P$-covering.\end{itemize}\end{lma}

\begin{proof}(a) Let us consider the following commutative square in $\Ee$
$$\xymatrix{\el_D(h^*(X))\ar[d]_{p_D(h^*(X))}\ar@{.>}[r]^{\bar{h}}&\el_B(X)\ar[d]^{p_B(X)}\\D\ar[r]_h&B}$$where $X$ is an object of $PB$ and $\bar{h}$ is induced by the identity of $h^*(X)$ in $PD$, cf. Lemma \ref{elements}. We claim that the square is a pullback in $\Ee$. Indeed, for any span $(f':D'\to \el_B(X),h':D'\to D)$ such that $p_B(X)f'=hh'$ we have to exhibit a unique map of spans towards $(\bar{h},p_D(h^*(X)))$. According to Lemma \ref{elements}, the existence of $f'$ amounts to a morphism $\star_{PD'}\to (hh')^*(X)$, but the latter amounts to a morphism $\star_{PD'}\to (h')^*(h^*(X))$ which, again according to Lemma \ref{elements}, yields a uniquely determined lift $D'\to\el_D(h^*(X))$ of $h'$. We have to check that this lift composed with $\bar{h}$ yields $f'$, but this just expresses that the latter is the lift of $hh'$ corresponding to $\star_{PD'}\to (hh')^*(X)$.

(b) The unit of the adjunction $c_B:\Ee/B\lrto PB:p_B$ at $f:A\to B$ is part of the following diagram
$$\xymatrix{
 A \ar[r]^{\eta_f}\ar[d]_f &  \el_B(f_!(\star_{PA}))\ar[ld]^{p_B(f_!(\star_{PA}))}\\
 B
}$$where we replaced $c_B(f)$ with its definition $f_!(\star_{PA})$. It suffices thus to show that $\eta_f$ is $P$-connected. Let us denote $\star$ the distinguished terminal object of $P\el_B(f_!(\star_{PA}))$ and write $p$ for $p_B(f_!(\star_{PA}))$. We have to show that the unique map $(\eta_f)_!(\star_{PA})\to\star$ is invertible. By consistency of the comprehension scheme, the left adjoint $p_!$ is fully faithful on morphisms with terminal codomain. Therefore, the image $p_!(\eta_f)_!(\star_{PA})\to p_!(\star)$ is the unique morphism in $PB$ with fixed domain and codomain and must be invertible because both sides are isomorphic to $f_!(\star_{PA})$. It follows that the given map $(\eta_f)_!(\star_{PA})\to\star$ is invertible as well.\end{proof}

\begin{thm}\label{main}Any consistent comprehension scheme $P$ on $\Ee$ induces a complete orthogonal factorisation system on $\Ee$ with left part consisting of $P$-connected morphisms and right part consisting of $P$-coverings.\end{thm}

\begin{proof}It follows from Lemma \ref{factorisation}b that each morphism factors as a $P$-connected morphism followed by a $P$-covering. Since $P$ is consistent, $P$-coverings compose by Proposition \ref{consistent} as do $P$-connected morphisms by their very definition. Both classes contain all isomorphisms so that it remains to be shown that any commuting square$$\xymatrix{A\ar[r]\ar[d]_l&C\ar[d]^r\\B\ar@{.>}[ru]\ar[r]&D}$$ with $P$-connected $l$ and $P$-covering $r$ admits a unique diagonal filler. By Lemma \ref{factorisation}a, the pullback $r':B\times_D C\to B$ exists in $\Ee$ and is a $P$-covering so that the factorisation system is complete. Diagonal fillers $B\to C$ correspond bijectively to sections $i':B\to B\times_D C$ of $r'$ such that $i'\circ l$ coincides with the comparison map $A\to B\times_DC$. By Lemma \ref{elements}, sections of $r'$ correspond bijectively to morphisms $\star_{PB}\to(r')_!(\star_{P(B\times_DC)})$ in $PB$. The comparison map $A\to B\times_CD$ corresponds to a uniquely determined morphism $l_!(\star_{PA})\to(r')_!(\star_{P(B\times_CD)})$ in $PB$. Since $l$ is $P$-connected we have an isomorphism $l_!(\star_{PA})\cong\star_{PB}$ yielding the unique section of $r'$ as required for the orthogonality of the factorisation system.\end{proof}

\begin{thm}\label{mainbis}Any complete factorisation system $(\Ll,\Rr)$ on $\Ee$ defines a consistent comprehension scheme $P_{(\Ll,\Rr)}$ on $\Ee$ assigning to an object $B$ the full subcategory $(\Ee/B)_\Rr$ of $\Ee/B$ spanned by the $\Rr$-morphisms with codomain $B$.

All consistent comprehension schemes inducing the factorisation system $(\Ll,\Rr)$ via Theorem \ref{main} are equivalent to $P_{(\Ll,\Rr)}$.\end{thm}

\begin{proof}For any morphism $f:A\to B$, the adjunction $f_!:(\Ee/A)_\Rr\lrto(\Ee/B)_\Rr:f^*$ is defined as follows: for $(a:A'\to A)\in(\Ee/A)_\Rr$ we define $f_!(a)=r_{f\circ a}\in(\Ee/B)_\Rr$, and for $(b:B'\to B)\in(\Ee/B)_\Rr$ we define $f^*(a)$ to be a pullback of $a$ along $f$, which in virtue of completeness exists and belongs to $(\Ee/A)_\Rr$. Both assigments are functorial in virtue of the orthogonality of the factorisation system. For adjointness, observe that as well morphisms $a\to f^*(b)$ in $(\Ee/A)_\Rr$ as well morphisms $f_!(a)\to b$ in $(\Ee/B)_\Rr$ correspond bijectively to commuting squares$$\xymatrix{A'\ar[r]\ar[d]_a&B'\ar[d]^b\\A\ar[r]_f&B}$$ in $\Ee$. To establish these bijective correspondences it is essential that $\Rr$-morphisms are left cancellable. This is a general property of orthogonal factorisation systems.

The category $(\Ee/B)_\Rr$ has the identity of $B$ as distinguished terminal object. The comprehension functor $c_B:\Ee/B\to(\Ee/\Bb)_\Rr$ is given by $f\mapsto r_f$, with adjoint $p_B:(\Ee/B)_\Rr\to\Ee/B$ the canonical embedding. In fact, the $(c_B,p_B)$-adjunction identifies $(\Ee/B)_\Rr$ with a full \emph{reflective} subcategory of $\Ee/B$. The comprehension scheme $P_{(\Ll,\Rr)}B=(\Ee/B)_\Rr$ is consistent in virtue of Proposition \ref{consistent} because $P_{(\Ll,\Rr)}$-coverings and $\Rr$-morphisms coincide and $\Rr$-morphisms are left cancellable.

Finally, any consistent comprehension scheme $P:\Ee\to\Adj_*$ inducing the factorisation system $(\Ll,\Rr)$ satisfies $P_{(\Ll,\Rr)}B=(\Ee/B)_\Rr=\Cov_B\simeq PB$.\end{proof}

\begin{rmk}We shall call the factorisation system ($P$-connected, $P$-covering) the \emph{comprehensive factorisation} defined by $P$. In the special case of small categories, the comprehension scheme $PA=\Set^A$ yields the factorisation of a functor into an initial functor followed by a discrete opfibration because $P$-connected functors are precisely initial functors, cf. \cite[Propositon 2]{StWa}. This is the factorisation originally introduced by Street and Walters as the comprehensive factorisation of a functor, cf. \cite[Theorem 3]{StWa}. The ``dual'' comprehension scheme $P'A=\Set^{A^\op}$ yields the factorisation of a functor into a final functor followed by a discrete fibration.

The powerset comprehension scheme on sets yields the image-factorisation of a set mapping. This example extends in a natural way to any well-powered regular category $\Ee$, using as powerset the set of subobjects ordered by inclusion. The comprehension scheme amounts here to the choice of a representing monomorphism $A\ito B$ for each subobject of $B$, and so affords some form of axiom of choice.

Our correspondence shows that \emph{all} complete orthogonal factorisation systems $(\Ll,\Rr)$ on $\Ee$ are ``comprehensive'' with respect to the scheme $P_{(\Ll,\Rr)}B=(\Ee/B)_\Rr$. Nevertheless, the freedom to choose equivalent comprehension schemes inducing the same factorisation system is valuable in practice. Moreover, the correspondence allows us to \emph{classify} factorisation systems according to properties of the corresponding comprehension scheme. For instance, the value of the comprehension scheme $P:\Ee\to\Adj_*$ at a terminal object $\star_\Ee$ of $\Ee$ contains a lot of information. We are primarily interested in those cases where $P(\star_\Ee)=\Set$ which actually fits best with our terminology ($P$-connected, $P$-covering).\end{rmk} 

The following proposition is remarkable insofar as it relates \emph{Frobenius reciprocity} for a comprehension scheme, as formulated by Lawvere \cite{Law1}, to a natural and often easy-to-check condition on the associated comprehensive factorisation system (which for precisely this relationship is sometimes called \emph{Frobenius property}). These two conditions are both equivalent to a third one, also frequently encountered in practice, and often called the \emph{Beck-Chevalley condition}.

\begin{prp}\label{Frobenius}For any consistent comprehension scheme $P$ on $\Ee$, the following three conditions are equivalent:
\begin{itemize}\item[(a)]For each $f:A\to B$ the adjunction $f_!:PA\lrto PB:f^*$ satisfies \emph{Frobenius reciprocity}, i.e. for any $X$ in $PA$ and $Y$ in $PB$, the canonical map$$f_!(X\times f^*(Y))\to f_!(X)\times Y$$is invertible.\item[(b)]\emph{(Beck-Chevalley)} For any pullback square in $\Ee$ with $P$-coverings $p$ and $q$ $$\xymatrix{A'\ar[r]^g\ar[d]_q&B'\ar[d]^p\\A\ar[r]_f&B}$$the induced natural transformation $g_!q^*\to p^*f_!$ is invertible. \item[(c)]$P$-connected morphisms are stable under pullback along $P$-coverings.\end{itemize}\end{prp}

\begin{proof}Let us first notice that in (b) we can assume $A'=\el_A(f^*(Y))$ and $B'=\el_B(Y)$, cf. the proof of Lemma \ref{factorisation}a. By consistency of $P$, we can furthermore replace $P\el_A(f^*(Y))$ with $PA/f^*(Y)$ and $P\el_B(Y)$ with $PB/Y$, cf. Lemma \ref{slice}, so that we get the following commutative square of categories$$\xymatrix{PA/f^*(Y)\ar[r]^{(\epsilon^Y_f)_!f_!}\ar[d]&PB/Y\ar[d]\\PA\ar[r]_{f_!}&PB}$$in which the vertical functors are the canonical projections. Now (b) is equivalent to the condition that for each $X$ in $PA$ the morphism $f_!(X\times f^*(Y)\to f^*(Y))$ composed with the counit $\epsilon^Y_f:f_!f^*(Y)\to Y$ is isomorphic (over $Y$) to the morphism $f_!(X)\times Y\to Y$ which is precisely condition (a). Condition (b) implies (c) by an easy diagram chase. Finally, if condition (c) holds, then the assignment just described for a $P$-connected morphism $f:A\to B$ and object $Y$ in $PB$ must take the identity of $f^*(Y)$ to a morphism isomorphic to the identity of $Y$. This means that $\epsilon^Y_f$ is invertible, i.e. $f^*$ is fully faithful for $P$-connected $f$. For such an $f$, condition (b) is automatically verified. By consistency of $P$, every morphism $f:A\to B$ factors as a $P$-connected morphism followed by a $P$-covering, cf. Lemma \ref{factorisation}b. It remains thus to show (a) or (b) for $P$-coverings $f$. For a $P$-covering $f$, condition (a) amounts to the familiar isomorphism $X\times_Z(Y\times Z)\cong X\times Y$.\end{proof}

\begin{rmk}The two comprehension schemes $P,P'$ on Cat satisfy the three conditions of Proposition \ref{Frobenius}. Condition (b) says that the square is \emph{exact} in the sense of Guitart \cite{Gu}. Note that condition (c) yields thus two stability properties which are dual to each other. It is remarkable that Guitart's characterisation of exact squares shows that (b) is an exact square precisely when for each $Y$ in $PB$ the induced $P$-covering $g/Y\to f/p(Y)$ is $P'$-connected. This can be used to give an alternative proof of (c): If $f$ is $P$-connected then $f/p(Y)$ is connected so that $g/Y$ is connected as well, which implies that $g$ is $P$-connected.\end{rmk}

\begin{sct}\textbf{Topological spaces.}\label{topological}We define a comprehension scheme $P_{top}$ for the full subcategory $\Top_{lsc}$ of the category of topological spaces spanned by the locally path-connected, semi-locally simply connected spaces: $P_{top}A$ is the category $\Sh_{loc}(A)$ of \emph{locally constant set-valued sheaves} on $A$. There is an equivalence of categories $\Sh_{loc}(A)\simeq\Cov(A)$ between locally constant sheaves on $A$ and topological coverings of $A$ (induced by the ``espace \'etal\'e'' construction $Et_A:\Sh(A)\to\Top/A$). The sheaf-theoretical restriction functor $f^*:P_{top}B\to P_{top}A$ corresponds to pulling back the corresponding covering, and the right adjoint $p_A=Et_A:\Sh_{loc}(A)\to\Top_{lsc}/A$ satisfies the universal property of Lemma \ref{elements}. The existence of the comprehension scheme $P_{top}$ hinges thus on the existence of a left adjoint $f_!:P_{top}A\to P_{top}B$.

For general set-valued sheaves such a left adjoint does not exist, but for locally constant sheaves over objects in $\Top_{lcs}$ it does. One uses that in this case the \emph{monodromy action} $\Sh_{loc}(A)\to\Set^{\Pi_1(A)}$ is an equivalence of categories so that $f_!:P_{top}A\to P_{top}B$ is induced by left Kan extension along the induced functor $\Pi_1(f):\Pi_1(A)\to\Pi_1(B)$ on fundamental groupoids. The quasi-inverse to the monodromy action assigns to a \emph{local system} $X:\Pi_1(A)\to\Set$ the presheaf whose sections over an open subset $U$ of $A$ consist of all families $(x_a\in X(a))_{a\in U}$ such that, for $(a,b)\in U\times U$ and $(\gamma:a\to b)\in\Pi_1(U)$, equality $X(\gamma)(x_a)=x_b$ holds. This presheaf is a locally constant sheaf on $A$ precisely because $A$ belongs to $\Top_{lsc}$.

The comprehension scheme $P_{top}$ is consistent because $P_{top}$-coverings coincide with topological coverings, and the latter compose and are left cancellable in $\Top_{lsc}$. In Section \ref{simplicial} we show that the $P_{top}$-connected morphisms are precisely those continuous maps which have connected homotopy fibres, i.e. which induce a bijection on path-components and a surjection on fundamental groups. The resulting comprehensive factorisation of a continuous map induces a formal construction of the universal covering space for any based space in $\Top_{lsc}$, cf. Section \ref{galois}.\end{sct}

\begin{dfn}\label{restriction}A comprehension scheme $P$ on $\Ee$ is said to \emph{restrict} to a full and replete subcategory $\Ee'$ of $\Ee$ if the restriction $P_{|\Ee'}$ is a comprehension scheme on $\Ee'$.\end{dfn}

According to Lemma \ref{elements}, a comprehension scheme $P$ restricts to $\Ee'$ precisely when for each object $A$ of $\Ee'$ and each object $X$ of $PA$, the object of elements $\el_A(X)$ belongs to $\Ee'$ or, equivalently, precisely when every $P$-covering of $\Ee$ with codomain in $\Ee'$ belongs to $\Ee'$. If $P$ is consistent then so is any of its restrictions.

\begin{sct}\label{groupoid}\textbf{Groupoids.} The full subcategory $\Gpd$ of $\Cat$ spanned by the groupoids permits a restriction of the comprehension scheme $P:\Cat\to\Adj_*$. The resulting $P$-coverings are the usual groupoid coverings, cf. Gabriel-Zisman \cite[Appendix I]{GZ}. Note that the two comprehension schemes $P,P'$ for $\Cat$ induce equivalent comprehension schemes for $\Gpd$. In particular, a functor between groupoids is initial (resp. a discrete opfibration) if and only if it is final (resp. a discrete fibration).

Bourn \cite{Bou} constructs the comprehensive factorisation for groupoids by a different method, available not only for groupoids in sets but more generally for the category $\Gpd(\Ee)$ of groupoids internal to any exact category $\Ee$. He considers $\Gpd(\Ee)$ as a full reflective subcategory of $\Ee^{\Delta^\op}$ (cf. Section \ref{simplicial}) and constructs (by means of the shift functor) for each groupoid $B$ in $\Ee$ a simplicial path-fibration $\Dec_.(B)\to B$. The comprehensive factorisation of $f:A\to B$ is then constructed by applying a ``fibrewise'' path-component functor to $A\times_{\Dec_.(A)}\Dec_.(B)\to\Dec_.(B)$.\end{sct}

\begin{dfn}\label{reflectivity}An \emph{adjunction} $i:\Dd\rlto\Ee:r$ is called \emph{$P$-reflecting} for a comprehension scheme $P$ on $\Dd$ if the left adjoint $r:\Ee\to\Dd$ induces slice functors $\Ee/B\to\Dd/r(B)$ with fully faithful right adjoint restrictions $\Cov_{r(B)}\to\Ee/B$.

If $i:\Dd\inc\Ee$ is a full embedding, $\Dd$ is called a \emph{$P$-reflective} subcategory of $\,\Ee$.\end{dfn}

A full reflective subcategory $\Dd$ of $\,\Ee$ is $P$-reflective \emph{if and only if} pullbacks of $P$-coverings exist in $\Ee$ along the components $\eta_B:B\to r(B)$ of the unit of the adjunction, and these pullbacks are preserved under the reflection. This means that for any $P$-covering $f':A'\to r(B)$ in $\Dd$ the following pullback $$\xymatrix{A\ar[r]\ar[d]& A'\ar[d]^{f'}\\B\ar[r]_{\eta_B}&r(B)}$$\emph{exists in $\Ee$} and has the property that the upper horizontal map is isomorphic to the unit-component $\eta_{A}:A\to r(A)$. Such a condition (for a specific choice of $P$-coverings) occurs at several places in literature. It is the key property of the \emph{reflective factorisation systems} of Cassidy-H\'ebert-Kelly \cite{CHK}.

\begin{prp}\label{extension}Let $P$ be a (consistent) comprehension scheme on $\Dd$.

If $\Dd$ is a full $P$-reflective subcategory of $\Ee$ then $P$ extends to a (consistent) comprehension scheme $P_\Ee$ on $\Ee$ putting $P_\Ee B=P(r(B))$. The $P_\Ee$-coverings are precisely those morphisms $f:A\to B$ constructible by a pullback square$$\xymatrix{A\ar[r]\ar[d]_f& A'\ar[d]^{f'}\\B\ar[r]_{\eta_B}&r(B)}$$in which $f': A'\to r(B)$ is a $P$-covering.

In the consistent case, the $P_\Ee$-connected morphisms $f$ are precisely those whose reflection $r(f)$ is $P$-connected.\end{prp}

\begin{proof}By $P$-reflectivity, the induced functor on slice categories $\Ee/B\to\Dd/r(B)$ has a right adjoint which is fully faithful at $P$-coverings. This implies that the comprehension functor $c_B:\Ee/B\to P_\Ee B$ has a fully faithful right adjoint so that $P_\Ee$ is a comprehension scheme. If $P$ is consistent, i.e. $P$-coverings compose and are left cancellable, then the same is true for $P_\Ee$-coverings by $P$-reflectivity. If this is the case, a morphism $f$ is $P_\Ee$-connected precisely when $f$ is left orthogonal to all $P_\Ee$-coverings. By adjunction this amounts to the condition that $r(f)$ is left orthogonal to all $P$-coverings, i.e. $r(f)$ is $P$-connected.\end{proof}

\begin{rmk}Proposition \ref{extension} and Theorem \ref{main} recover one of the main results of Cassidy-H\'ebert-Kelly \cite{CHK}, namely : assume that for a given stable, composable and left cancellable class $\Rr$ of morphism in $\Ee$ the following two conditions hold:\begin{enumerate}\item the full subcategory $\Dd$ of $\Ee$ spanned by the objects $B$ such that $B\to\star_\Ee$ belongs to $\Rr$ is reflective in $\Ee$;\item The class $\Ll$ of those morphisms which are inverted by the reflection (1) is closed under pullback along morphisms in $\Rr$.\end{enumerate}Then $(\Ll,\Rr)$ is a complete orthogonal factorisation system on $\Ee$.

Indeed, the reflective subcategory $\Dd$ is equipped with the consistent comprehension scheme $PB=\Dd/B$. Since the unit-components of the adjunction are inverted by the reflection, condition (2) implies that $\Dd$ is $P$-reflective. Therefore $P$ extends to a consistent comprehension scheme $P_\Ee$ inducing a complete orthogonal factorisation system on $\Ee$. $P_\Ee$-coverings coincide with $\Rr$-morphisms while $P_\Ee$-connected morphisms are those whose reflection is $P$-connected, i.e. invertible.\end{rmk}

\begin{rmk}\label{fibredreflection}For a given full reflective subcategory $\Dd$ of $\Ee$ with comprehension scheme $P$ it may be difficult to check $P$-reflectivity using Definition \ref{reflectivity}. In view of Proposition \ref{extension} and Theorem \ref{main}, the resulting class of $P_\Ee$-coverings is \emph{stable} (i.e. pullbacks along arbitrary maps in $\Ee$ exist and are $P_\Ee$-coverings) and, for each $P_\Ee$-covering $f:A\to B$, the reflection $r(f)$ is a $P$-covering and the naturality square$$\xymatrix{A\ar[r]^{\eta_A}\ar[d]_f&r(A)\ar[d]^{r(f)}\\B\ar[r]_{\eta_B}&r(B)}$$is cartesian. Conversely, if the class of those morphisms $f:A\to B$ for which $r(f)$ is a $P$-covering and the naturality square is cartesian, forms a \emph{stable} class of morphisms of $\Ee$, then $\Dd$ is $P$-reflective and the stable class represents precisely the $P_\Ee$-coverings. It is often easier to check $P$-reflectivity using this second method.\end{rmk}

\begin{sct}\textbf{Simplicial sets.}\label{simplicial} Through the nerve functor the category $\Gpd$ of groupoids is a full reflective subcategory of the category $\widehat{\Delta}=\Set^{\Delta^\op}$ of simplicial sets. The reflection $\Pi_1:\widehat{\Delta}\to\Gpd$ is usually called the \emph{fundamental groupoid functor}. The subcategory of groupoids is $P$-reflective in $\widehat{\Delta}$ with respect to the comprehension scheme $P$ of Section \ref{groupoid} as follows from Remark \ref{fibredreflection} applied to the stable class of \emph{discrete Kan fibrations}. Indeed, for any discrete Kan fibration $f:A\to B$, the induced functor $\Pi_1(f):\Pi_1(A)\to\Pi_1(B)$ is a $P$-covering of groupoids, and the naturality square is cartesian because the comparison map $A\to B\times_{\Pi_1(A)}\Pi_1(B)$ is bijective on $0$-simplices, and hence invertible \cite[Appendix I, Proposition 2.4.2]{GZ}.

According to Proposition \ref{extension} there is an extended comprehension scheme $P_{\widehat{\Delta}}$. The latter induces the usual covering theory for simplicial sets, cf. Gabriel-Zisman \cite[Appendix I.2-3]{GZ}. As we have seen, the $P_{\widehat{\Delta}}$-coverings are discrete Kan fibrations. The associated comprehensive factorisation of a simplicial map yields in particular the universal covering for any based simplicial set, cf. Section \ref{galois}.

The adjunction $|\!-\!|:\widehat{\Delta}\lrto\Top_{lsc}:\Sing$ has the property that both, the geometric realisation functor and the singular functor, preserve coverings. Since the counit of the adjunction is a cartesian natural transformation when restricted to coverings, an orthogonality argument shows that the geometric realisation functor takes $P_{\widehat{\Delta}}$-connected simplicial maps to $P_{top}$-connected continuous maps. In other words, geometric realisation preserves comprehensive factorisations.

A continuous map $f$ has connected homotopy fibres if and only if the induced map $\Pi_1(f)=\Pi_1(\Sing(f))$ on fundamental groupoids has connected homotopy fibres.  Quillen's Theorem B \cite{Q} combined with \cite[Proposition 2]{StWa} shows that a map of groupoids has connected homotopy fibres if and only if it is $P$-connected.  In virtue of Proposition \ref{extension}, the analogous statement is true for a map of simplicial sets, resp. a continuous map of topological spaces.\end{sct}

\section{Feynman categories and multicategories}\label{feynman}

Hermida characterises in \cite{H} monoidal categories as special non-symmetric multicategories, namely as the representable one's. The idea of Feynman categories \cite{KW} is somehow opposite, namely to consider multicategories as special symmetric monoidal categories. This second point of view yields a parallel understanding of the standard comprehension schemes for Feynman categories and multicategories.

An important role is played by \emph{permutative categories} \cite{May} (which are symmetric \emph{strict} monoidal categories) because the \emph{free} permutative category $\Vv^\otimes$ generated by a category $\Vv$ admits a useful explicit description (cf. e.g. \cite{EM}). In particular, if $\Vv$ is a groupoid then so is $\Vv^\otimes$. For any category $\Cc$, we denote by $\Cc_{iso}$ the subcategory of invertible morphisms. If $\Cc$ is symmetric monoidal, then so is $\Cc_{iso}$.

We shall call a symmetric monoidal category $\FF$ \emph{framed} if there is a groupoid $\Vv$ equipped with a full embedding $\iota:\Vv\inc\FF$ such that the induced functor $\Vv^\otimes\to\FF_{iso}$ is an equivalence of symmetric monoidal categories. In particular, any framed symmetric monoidal category has an essentially small underlying category.

Framed symmetric monoidal categories are thus triples $(\FF,\Vv,\iota)$. They form a category with morphisms $(\FF_1,\Vv_1,\iota_1)\to(\FF_2,\Vv_2,\iota_2)$ the pairs $(\phi,\psi)$ consisting of a strong symmetric monoidal functor $\phi:\FF_1\to\FF_2$ and a map of groupoids $\psi:\Vv_1\to\Vv_2$ such that $\phi\iota_1=\iota_2\psi$.

Any framed symmetric monoidal category $(\FF,\Vv,\iota)$ induces a small multicategory $\Oo_\FF$ (aka coloured symmetric operad) with same objects as $\Vv$ and multimorphisms $$\Oo_\FF(v_1,\dots,v_k;v)=\FF(\iota(v_1)\otimes\cdots\otimes\iota(v_k),\iota(v)).$$The groupoid $\Vv$ coincides with the groupoid of invertible unary morphisms of $\Oo_\FF$. Conversely, any small multicategory $\Oo$ induces a framed symmetric monoidal category $(\FF_\Oo,\Vv_\Oo,\iota_\Oo)$: the groupoid $\Vv_\Oo$ is the groupoid of invertible unary morphisms of $\Oo$, the objects of $\FF_\Oo$ are those of $(\Vv_\Oo)^\otimes$, written as tensor products of objects of $\Vv$, and the morphisms of $\FF_\Oo$ are given by\begin{equation*}\FF_\Oo(v_1\otimes\cdots\otimes v_k,w_1\otimes\cdots\otimes w_l)
=\coprod_{\phi:\{1,\dots,k\}\to\{1,\dots,l\}}\Oo(v_{\phi^{-1}(1)};w_1)\times\cdots\times\Oo(v_{\phi^{-1}(l)};w_l)\end{equation*}where for any ordered subset $I=(i_1<\dots<i_r)$ of $\{1,\dots,k\}$, the symbol $v_I$ stands for the sequence $v_{i_1},\dots,v_{i_r}$.

The assignment of a symmetric monoidal category to a multicategory occurs at several places in literature. The one-object case goes back to May-Thomason \cite{MT}. The formula above occurs in Elmendorf-Mandell \cite[Theorem 4.2]{EM}. Hermida \cite{H} uses a similar functor from non-symmetric multicategories to monoidal categories. The idea of bookkeeping a ``framing'' goes back to Getzler's ``patterns'' \cite{Ge}, where the functor $\iota^\otimes:\Vv^\otimes\to\FF$ is only supposed to be essentially surjective.

The two assignments $\Oo\mapsto\FF_\Oo$ and $\FF\mapsto\Oo_\FF$ form an adjunction between small multicategories and framed symmetric monoidal categories. For each small multicategory $\Oo$ the unit-component $\Oo\to\Oo_{\FF_\Oo}$ is invertible. Small multicategories form thus a full coreflective subcategory of framed symmetric monoidal categories.

We arrive at the following reformulation of the definition of a Feynman category of \cite{KW}: A \emph{Feynman category} is a framed symmetric monoidal category $(\FF,\Vv,\iota)$ which is \emph{hereditarily framed} in the sense that the double slice category $\FF\downarrow\FF$ itself is a framed symmetric monoidal category with respect to the ``groupoid'' $(\FF\downarrow\Vv)_{iso}$, i.e. the canonical map $(\FF\downarrow\Vv)^\otimes_{iso}\to(\FF\downarrow\FF)_{iso}$ is an equivalence of symmetric monoidal categories. Notice that our smallness condition ($\Vv$ small) is slightly more restrictive than the one used in \cite{KW}.

\begin{prp}The counit-component $\FF_{\Oo_\FF}\to\FF$ is an equivalence of framed symmetric monoidal categories precisely when $\,\FF$ is a Feynman category.\end{prp}

This proposition follows from \cite[Section 1.2, Remark 1.4.2 and Section 1.8.3]{KW}. It also follows from a general statement of Batanin, Kock and Weber about \emph{pinned} symmetric monoidal categories, cf. \cite[Proposition 4.2, Theorems 5.13 and 5.15]{BKW}. In particular, the $2$-categories of small multicategories and of Feynman categories are $2$-equivalent. This $2$-equivalence respects the respective notions of \emph{algebra}.

An algebra for a Feynman category $\FF$ or, as we shall say, an \emph{$\FF$-operad} is a strong symmetric monoidal functor $\FF\to(\Set,\times,\star_\Set)$. $\FF$-operads and symmetric monoidal natural transformations form a locally finitely presentable category $\FF\opd$, cf. Getzler \cite[Theorem 2.10]{Ge}. Therefore, Freyd's Adjoint Functor Theorem applies, and each Feynman functor $f:(\FF,\Vv,\iota)\to(\FF',\Vv',\iota')$ has a limit-preserving restriction functor $f^*:\FF'\opd\to\FF\opd$ which comes equipped with a left adjoint extension functor $f_!:\FF\opd\to\FF'\opd$, see also \cite[Theorem 1.6.2]{KW}.

\begin{rmk}\label{hereditary}It is fundamental that all these extension functors are given by \emph{pointwise left Kan extension}, i.e. for any $\FF$-operad $F$ and object $Y$ of the target $\FF'$, the extension $f_!(F)$ at $Y$ is given by $(f_!F)(Y)=\colim_{f(-)\downarrow Y}F(-)$ where the colimit is computed in sets. This property is one of the main advantages of Feynman categories over multicategories.

Let us sketch the argument: since for any functor $\Vv\to\Vv'$, extension along the induced functor of permutative categories $\Vv^\otimes\to(\Vv')^\otimes$ is a pointwise left Kan extension, it suffices to show that for any Feynman category $(\FF,\Vv,\iota)$, extension along $\iota^\otimes:\Vv^\otimes\to\FF$ is a pointwise left Kan extension. This amounts to showing that pointwise left Kan extension takes permutative functors $\Vv^\otimes\to\Set$ to strong symmetric monoidal functors $\FF\to\Set$. This in turn can be reduced to the following property: for each decomposition $X\cong v_1\otimes\cdots\otimes v_k$ of an object $X$ of $\FF$ into a tensor product of objects of $\Vv$, the canonical map $$(\FF/v_1)_{iso}\times\cdots\times(\FF/v_k)_{iso}\to(\FF/X)_{iso}$$ is a final functor, i.e. has connected coslices. This last condition is a reformulation of the hereditary condition of $\FF$, cf. \cite[Section 1.8.5]{KW}.

Let us mention that Batanin, Kock and Weber establish the following converse statement, cf. \cite[Propositions 2.11, 3.14, 4.2]{BKW}: if for a framed symmetric monoidal category $(\FF,\Vv,\iota)$, extension along $\iota^\otimes:\Vv^\otimes\to\FF$ is given by pointwise left Kan extension, then $\FF$ is hereditarily framed, i.e. a Feynman category.\end{rmk}

\begin{prp}There is a consistent comprehension scheme for Feynman categories assigning to a Feynman category $\FF$ the category of set-valued $\FF$-operads.\end{prp}

\begin{proof}We first exhibit a \emph{Feynman category of elements} $\el_\FF(F)$ over $\FF$ with the universal property of Lemma \ref{elements} for each $(\FF,\Vv,\iota)$-operad $F$, closely following \cite{KL}. Indeed, the usual category of elements $\el(F)$ of the underlying diagram $F:\FF\to\Set$ comes equipped with a Feynman category structure: for objects $(X,x\in F(X))$ and $(Y,y\in F(Y))$, the tensor is given by $(X,x)\otimes(Y,y)=(X\otimes Y,\phi^{X,Y}_F(x,y))$ where $\phi^{X,Y}_F:F(X)\times F(Y)\cong F(X\otimes Y)$ is the symmetric monoidal structure of $F$. This endows $\el(F)$ with the structure of a symmetric monoidal category. Moreover $\el(F\iota)$ is a groupoid (cf. Section \ref{groupoid}) equipped with a full embedding $\el(\iota):\el(F\iota)\inc\el(F)$. Since we have isomorphisms $\el(F\iota)^\otimes\cong\el(F\iota^\otimes)$ and $\el(F)_{iso}\cong\el(F_{|\FF_{iso}})$, we get a framed symmetric monoidal category $\el_\FF(F)=(\el(F),\el(F\iota),\el(\iota))$ over $\FF$.

Since the projection $\el_\FF(F)\to\FF$ is a discrete opfibration, the hereditary condition of $\FF$ (as formulated in Remark \ref{hereditary}) lifts to $\el_\FF(F)$, see \cite{KL} for a detailed proof. The latter is thus a Feynman category over $\FF$. Since extensions along Feynman functors are computed as pointwise left Kan extensions, the fact that the usual category of elements construction defines a consistent comprehension scheme on $\Cat$ (cf. Section \ref{Cat}) implies that the Feynman category of elements construction defines a consistent comprehension scheme for Feynman categories.\end{proof}

\begin{rmk}Consistency of the comprehension scheme implies that for any $\FF$-operad $F$, the category of $\FF$-operads over $F$ is canonically equivalent to the category of $\el_\FF(F)$-operads, cf. Lemma \ref{slice}. Moreover, any Feynman functor $f:\FF\to\FF'$ factors as a connected functor $\FF\to\el_{\FF'}(f_!(\star_{\Fop}))$ followed by a covering projection $\el_{\FF'}(f_!(\star_{\Fop}))\to\FF'$ and this factorisation is unique up to isomorphism. The existence of such a factorisation has been proved in \cite{KL}, but its uniqueness is new.

The same statements hold for small multicategories $\Oo$ by restriction of the comprehension scheme, cf. Definition \ref{restriction}. Any $\Oo$-algebra $A$ defines a \emph{multicategory of elements} $\el_\Oo(A)$ with objects pairs $(X\in\Ob\Oo, x\in A(X))$ and multimorphisms $$\el_\Oo(A)((X_1,x_1),..,(X_k,x_k);(X,x))=\{f\in\Oo(X_1,..,X_k;X)\,|\,A(f)(x_1,..,x_k)=x\}.$$ The resulting equivalence of categories $\el_\Oo(A)\textrm{-}\Alg\simeq\Oo\textrm{-}\Alg/A$ is folklore but our proof seems to be the first written account of it. The comprehensive factorisation of a multifunctor extends Street and Walters' comprehensive factorisation of a functor.\end{rmk}

\begin{rmk}In Section \ref{scheme} we developed comprehension schemes and comprehensive factorisations for set-based categories considering the cartesian product as symmetric monoidal structure. Parts of the theory extend to comprehension schemes taking values in symmetric monoidal categories and adjunctions with symmetric lax comonoidal left adjoint and symmetric lax monoidal right adjoint. This leads to a category of elements construction for symmetric lax (co)monoidal functors. Some steps in this direction are made in \cite[Section 3.2]{KW} and \cite[Section 2.2]{KL}.\end{rmk}

\begin{sct}\textbf{Planar-cyclic and surface-modular operads.} Cyclic and modular operads have been introduced by Getzler and Kapranov \cite{GK} as tools to understand moduli spaces of surfaces and algebraic curves. Since their introduction they have proved useful in other areas of mathematics as well, e.g. in combinatorics, in computer science or in mathematical physics. One of our motivations in writing this text on comprehensive factorisations was a recent article of Markl \cite{Ma} in which he defines a new class of modular-like operads, based on the combinatorics of ``polycylic orderings'', with the advantage over modular operads of having less built-in symmetries. With the comprehensive factorisation in hand, we shall see that Markl's definition is a very natural one and to some extent the only possible.

We shall consider the following commutative diagram of Feynman categories

 \begin{equation*}
\label{modulardiag}
\xymatrix{\FF_{\neg\,\mathrm{sym}}\ar[rr]^{i'}\ar[d]^{p(\tau_{assoc})}&&\FF_{\neg\,\mathrm{cyc}} \ar[rr]^{k'} \ar[d]^{p(\tau_{planar})} && \FF_{\neg\,\mathrm{mod}}\ar[d]^{p(\tau_{ribbon})} \\
\FF_{sym}\ar[rr]_i&&\Fcyc \ar[rr]_k\ar[drr]_j && \Fmod\ar[d]^{p(\tau_{genus})}\\&&&&\FF_{ctd}}
\end{equation*}in which the horizontal Feynman functors $i',k',k$ are connected and \emph{all} vertical Feynman functors are coverings. In virtue of the uniqueness of comprehensive factorisations, the whole diagram is entirely determined by the Feynman functors $i$ and $j$ together with the $\FF_{sym}$-operad $\tau_{assoc}$. In particular, we have the identifications $j_!(\star_{\Fcyc\opd})=\tau_{genus},\quad i_!(\tau_{assoc})=\tau_{planar}$ and $k_!(\tau_{planar})=\tau_{ribbon}$.

The Feynman category $\FF_{sym}$ has as objects sequences $(n_1,\dots,n_k)$ of natural numbers which we identify with disjoint unions of corollas $\star_{n_1}\sqcup\cdots\sqcup\star_{n_k}$ having $n_i$ flags\footnote{In our context, a graph is given by a quadruple $(V,F,s,\iota)$ where $V$ is a set of vertices, $F$ a set of abstract flags, $s:F\to V$ the source-map and $\iota:F\to F$ an involution. A fixpoint under $\iota$ is called a flag, a non-fixpoint a half-edge. The orbits formed by two half-edges are called edges. Each graph can be topologised in such a way that edges become homeomorphic to $[0,1]$ or $S^1$ and flags homeomorphic to $[0,1[$. A corolla is a connected graph without edges.} respectively. The generating morphisms $\star_{n_1}\sqcup\cdots\sqcup\star_{n_k}\to\star_n$ are represented by rooted trees having $k$ vertices and $n$ flags such that each source-corolla is identified with the open neighborhood of a specific vertex of the tree and such that the target-corolla is identified with the tree itself after contraction of all its edges. It is important that these generating morphisms are represented not just by abstract trees, but by trees having \emph{all} their half-edges (resp. flags) identified with one and exactly one flag of the source (resp. target). The symmetric monoidal structure of $\FF_{sym}$ is given by disjoint union, i.e. general morphisms are represented by rooted forests. Composition of generating morphisms corresponds to insertion of one rooted tree into a specific vertex of another rooted tree (cf. either \cite[Part IV]{BB} or \cite[Appendix]{KW} for precise definitions). The underlying multicategory of $\FF_{sym}$ is isomorphic to the $\NN$-coloured symmetric operad of \cite[1.5.6]{BM} whose algebras are symmetric operads. Therefore, $\FF_{sym}$-operads are symmetric operads as well.

Every class $\Gamma$ of graphs which is closed under the process of inserting a graph of $\Gamma$ into the vertex of another graph of $\Gamma$ gives rise to a well-defined Feynman category $\FF^\Gamma$, and hence also to a well-defined multicategory $\Oo^\Gamma$. The Feynman category $\FF_{sym}$ corresponds thus to the insertional class of rooted trees.

The Feynman category $\Fcyc$ corresponds to the insertional class of general unrooted trees. The $\Fcyc$-operads are precisely the cyclic operads of Getzler-Kapranov. The Feynman functor $i:\FF_{sym}\to\Fcyc$ is defined by assigning to a rooted tree its underlying unrooted tree. This increases ``symmetry'' because the symmetry group of a rooted corolla $\star_{n+1}$ is $\Sigma_n$ while the symmetry group of $i(\star_{n+1})$ is $\Sigma_{n+1}$.

Connected graphs form an insertional class of graphs to which corresponds the Feynman category $\FF_{ctd}$, cf. \cite{KWZ}, where this Feynman category has been denoted $\mathfrak{G}^{ctd}$. Since trees are connected we have a Feynman functor $j:\FF_{cyc}\to\FF_{ctd}$. This Feynman functor is \emph{not} connected because the extension of a terminal cyclic operad along $k$ yields the $\FF_{ctd}$-operad $\tau_{genus}$ which assigns to each corolla $\star_{n_i}$ of $\FF_{ctd}$ the set $\NN$ of natural numbers, and to each generating morphism of $\FF_{ctd}$ the operation of adding the \emph{genus} of the representing connected graph, cf. \cite[5.4.2]{KL}. The genus of a connected graph is by definition the rank of its fundamental group (which equals the number of edges not belonging to a spanning subtree).

The Feynman category $\FF_{mod}$ is the \emph{Feynman category of elements} of $\tau_{genus}$. Its operads are precisely the modular operads of Getzler-Kapranov \cite{GK}, while $\FF_{ctd}$-operads are modular operads ``without genus-labeling''. In other words, the comprehensive factorisation of $j:\Fcyc\to\FF_{ctd}$ yields the connected Feynman functor $k:\Fcyc\to\Fmod$ followed by the covering $p(\tau_{genus}):\Fmod\to\FF_{ctd}$. It is thus the genus-labeling of a modular operad which is responsible for the connectedness of $k$.

Let us now define the upper horizontal line. The $\FF_{sym}$-operad $\tau_{assoc}$ is the symmetric operad for \emph{associative monoids}. The latter associates to a rooted corolla $\star_{n+1}$ the symmetric group $\Sigma_n$ on $n$ letters. The elements of this symmetric group can be thought of as orderings of the non-root flags of $\star_{n+1}$. It follows that the value of $\tau_{assoc}$ at a generating morphism of $\FF_{sym}$ is the set of (isotopy classes of) planar embeddings of the representing rooted tree. Therefore, the Feynman category $\FF_{\neg\,\textrm{sym}}$ of elements of $\tau_{assoc}$ is equivalent to the Feynman category associated with the insertional class of planar rooted trees. Here, all symmetry groups are trivial and $\FF_{\neg\,\textrm{sym}}$-operads are precisely \emph{non-symmetric} operads.

In order to get the Feynman category $\FF_{\neg\,\mathrm{cyc}}$ we have to compute the ``cyclic envelope'' $i_!(\tau_{assoc})$ which we denote $\tau_{planar}$. Indeed, the latter assigns to a generating morphism of $\Fcyc$ the set of planar structures of its representing tree. As above, this implies that the Feynman category $\FF_{\neg\,\mathrm{cyc}}$ of elements of $\tau_{planar}$ is equivalent to the Feynman category associated with the insertional class of planar trees. This time there are non-trivial symmetry groups. For instance, a planar corolla $\star_{n+1}$ admits the cyclic group of order $n+1$ as symmetry group. We call the associated $\FF_{\neg\,\mathrm{cyc}}$-operads \emph{planar-cyclic} operads. Markl calls them non-$\Sigma$-cyclic operads.

Finally, in order to get the last Feynman category $\FF_{\neg\,\mathrm{mod}}$ we have to compute the ``modular envelope'' $j_!(\tau_{planar})$ which we denote $\tau_{ribbon}$. Although the computation of this modular envelope is quite involved, cf. \cite{BK}, the result is easy to state: one obtains for each genus-labelled corolla $\star_{g,n}$ of $\Fmod$ the set $\tau_{ribbon}(\star_{g,n})$ of \emph{equivalence classes} of one-vertex \emph{ribbon graphs} with $g$ loops and $n$ flags. These equivalence classes correspond one-to-one to polycyclic orderings of the set of flags into $b$ possibly empty cycles with the additional property that $g-b+1$ is even and nonnegative. It can be checked that $\FF_{\neg\,\mathrm{mod}}$-operads are precisely Markl's geometric non-$\Sigma$-modular operads. We call them \emph{surface-modular} operads.

A \emph{ribbon graph} is a graph $(V,F,s,\iota)$ together with cyclic orderings of the fibres $s^{-1}(v),\,v\in V$. These cyclic orderings assemble into a permutation $N:F\to F$ whose cycles are precisely the fibres of $s:F\to V$. Two ribbon graphs are equivalent if there exists a third ribbon graph which ``ribbon contracts'' to both, where ``ribbon contraction'' means contraction of a subforest. Equivalence classes of ribbon graphs correspond one-to-one to topological types of bordered oriented surfaces where the boundary components of the surface correspond to the cycles of $N_\infty=N\circ\iota$. The flags contained in such a cycle give rise to markings of the corresponding boundary component of the surface. Empty cycles correspond to empty boundaries and are usually considered as punctures of the surface. Under this correspondence the nonnegative integer $\frac{1}{2}(g-b+1)$ is the genus of the associated surface. The result above gives thus an explicit link between the combinatorics of surface-modular operads and the topological classification of bordered oriented surfaces, cf. \cite{BK}.\end{sct}

\section{Galois theory for categories with discrete comprehension scheme}\label{galois}

That Galois theory for field extensions is intimately related to covering theory for spaces has been advocated by Grothendieck \cite{Gr} and since then by many others. We try to follow this line by developing some pieces of Galois covering theory in the general context of categories equipped with a discrete comprehension scheme.

Throughout this section we fix a category $\Ee$ with consistent comprehension scheme $P$ and terminal object $\star_\Ee$ and assume that $P(\star_\Ee)$ is the \emph{category of sets}. We assume furthermore that for each $f:A\to B$ the adjunction $f_!:PA\lrto PB:f^*$ satisfies \emph{Frobenius reciprocity} or, equivalently, that pullbacks of $P$-connected morphisms along $P$-coverings are again $P$-connected, cf. Proposition \ref{Frobenius}. A consistent comprehension scheme with these two properties will be called \emph{discrete}. We shall omit $P$ from notation.

For any object $A$ we define the object $\pi_0(A)$ of connected components of $A$ by comprehensive factorisation $A\to\pi_0(A)\to\star_\Ee$ of the unique map $A\to\star_\Ee$.

An object $A$ is called \emph{discrete} (resp. \emph{connected}) if $A\to\pi_0(A)$ (resp. $\pi_0(A)\to\star_\Ee$) is invertible. A morphism $f:A\to B$ is called \emph{coherent} if $f^*:PB\to PA$ preserves coproducts.
An object $A$ is called \emph{locally connected} if $A\to\pi_0(A)$ is coherent.

\begin{prp}\label{locallyconnected}Small coproducts of copies of $\star_\Ee$ exist and are precisely the discrete objects of $\Ee$. Every locally connected object is coproduct of connected objects and this decomposition is stable under pullback along coherent maps.\end{prp}

\begin{proof}By definition, the discrete objects are precisely those covering $\star_\Ee$. The category $\Cov_{\star_\Ee}$ is equivalent to $P(\star_\Ee)=\Set$ where every object is a coproduct of singletons. Therefore, every discrete object of $\Ee$ is a coproduct of copies of $\star_\Ee$. The injections of this coproduct are coverings by Proposition \ref{discrete}b. The Frobenius property implies then that each element $i:\star_\Ee\to\pi_0(A)$ induces a \emph{connected} subobject $A_i$ of $A$ by pullback along $A\to\pi_0(A)$. If $A\to\pi_0(A)$ is coherent we get a canonical isomorphism $A\cong\coprod_{i\in\pi_0(A)}A_i$. Stability under coherent pullback follows from the way pullbacks of coverings are constructed, cf. the proof of Lemma \ref{factorisation}a.\end{proof}

\begin{lma}\label{discrete}--

\begin{itemize}\item[(a)]The discrete objects form a full reflective subcategory with reflection $\pi_0$.\item[(b)]Any morphism between discrete objects is a covering.\item[(c)]Any connected morphism $A\to B$ induces a bijection $\pi_0(A)\to\pi_0(B)$.\end{itemize}\end{lma}

\begin{proof}(a) The required universal property of $A\to\pi_0(A)$ follows from orthogonality $$\xymatrix{A\ar[r]\ar[d]&D\ar[d]\\\pi_0(A)\ar[r]\ar@{.>}[ru]&\star_\Ee}$$where $D$ is a discrete object, i.e. $D\to\star_\Ee$ is a covering.

(b) This follows from left cancellability of coverings.

(c) Right cancellability of connected morphisms implies that $\pi_0(A)\to\pi_0(B)$ is connected. By (b) the latter is also a covering and hence invertible.\end{proof}

A covering $A\to B$ is called an \emph{epi}covering (resp. \emph{mono}covering) if the induced mapping $\pi_0(A)\to\pi_0(B)$ is surjective (resp. injective). A morphism $A\to B$ is called \emph{complemented} if the comparison map $A\to \pi_0(A)\times_{\pi_0(B)}B$ is invertible.

\begin{prp}\label{epimono}Every covering factors essentially uniquely into an epicovering followed by a complemented monocovering.  If the codomain is locally connected the latter is the inclusion of a coproduct of connected components.\end{prp}

\begin{proof}For a covering $A\to B$ consider the following commutative diagram
$$\xymatrix{A\ar[rr]^i\ar[d]&&D\times_{\pi_0(B)} B \ar@{->}[rr]^j\ar[d]^k &&B\ar[d] \\
\pi_0(A)\ar@{->}[rr]&&D\ar@{->}[rr]&& \pi_0(B)}$$in which the lower line is the ``image factorisation'' of $\pi_0(A)\to\pi_0(B)$. Since the inclusion $D\hookrightarrow\pi_0(B)$ is a covering, its pullback $j$ exists and is a covering, and hence $i$ is a covering as well, by left cancellability of coverings. The Frobenius property implies that $k$ is connected so that the discrete object $D$ gets identified with $\pi_0(D\times_{\pi_0(B)}B)$. The upper line is thus the required factorisation. Essential uniqueness amounts to orthogonality between epicoverings and complemented monocoverings inside the category of coverings. This follows from a diagram chase using Lemma \ref{discrete}a and the orthogonality between surjections and injections in $\Set$.

If $B$ is locally connected then $D\times_{\pi_0(B)} B$ is coproduct of those connected components of $B$ which are indexed by elements of $D$, cf. proof of Proposition \ref{locallyconnected}.\end{proof}

\begin{cor}\label{complement}The following three conditions are equivalent:\begin{itemize}\item[(E)]Epicoverings are strongly epimorphic inside the category of coverings;\item[(M)]Monomorphic coverings are complemented;\item[(R)]Every covering factors into a strongly epimorphic covering followed by a complemented monocovering.\end{itemize}\end{cor}

\begin{proof}According to Lemma \ref{discrete}a-b the discrete objects span a full reflective subcategory of the category of coverings. The reflection $\pi_0$ takes a strongly epimorphic covering to a surjection, i.e. every strongly epimorphic covering is an epicovering. Condition (E) expresses thus that inside the category of coverings strong epimorphisms and epicoverings coincide. Similarily, condition (M) expresses that monomorphic coverings and complemented monocoverings coincide. Since by Proposition \ref{epimono} epicoverings and complemented monocoverings form orthogonal classes in the category of coverings, conditions (E), (M) and (R) are equivalent.\end{proof}

An object $A$ is called \emph{based} if it comes equipped with a morphism $\alpha:\star_\Ee\to A$ in which case we shall write $(A,\alpha)$.

\begin{dfn}The \emph{universal covering} of a based object $(A,\alpha)$ is defined by comprehensive factorisation $\star_\Ee\to\Uu_{\alpha}\to A$ of $\alpha$. The fundamental group $\pi_1(A,\alpha)$ is the group of automorphisms of $u_{\alpha}:\Uu_{\alpha}\to A$ fixing $A$.\end{dfn}

It follows from the orthogonality of the comprehensive factorisation that for each covering $p:(B,\beta)\to(A,\alpha)$ there is one and only one lift of coverings $\Uu_{\alpha}\to B$ $$\xymatrix{\star_\Ee\ar[r]^\beta\ar[d]_{\alpha'}&B\ar[d]^p\\\Uu_{\alpha}\ar@{.>}[ru]\ar[r]_{u_{\alpha}}&A}$$taking $\alpha'$ to $\beta$. This justifies our terminology.

Since by Lemma \ref{elements}, the fundamental group $\pi_1(A,\alpha)$ can also be identified with the automorphism group of $\alpha_!(\star_\Set)$ in $PA$, based maps $f:(A,\alpha)\to(B,\beta)$ induce group homomorphisms $\pi_1(f):\pi_1(A,\alpha)\to\pi_1(B,\beta)$ in a functorial way.

\begin{rmk}The previous definitions recover the classical $\pi_0$- and $\pi_1$-functors for topological spaces, simplicial sets and groupoids with respect to the comprehension schemes discussed in Section \ref{scheme}.  Although this Galois-type definition of fundamental group a priori depends on the choice of basepoint we will see below that under certain conditions (essentially those of Corollary \ref{complement}) different basepoints of a connected, locally connected object yield isomorphic fundamental groups.

For the category of small categories we get the usual $\pi_0$-functor, but the conditions of Corollary \ref{complement} are not met and different basepoints yield here in general non-isomorphic universal coverings and non-isomorphic fundamental groups. Moreover, every based category $(A,\alpha)$ has two natural ``dual'' fundamental groups, the automorphism group of the universal $P$-covering $\alpha/A\to A$, and the automorphism group of the universal $P'$-covering $A/\alpha\to A$, cf. Section \ref{Cat}.\end{rmk}

\begin{prp}\label{monadicity}For any based object $(A,\alpha)$, the fibre functor $\alpha^*:PA\to\Set:\alpha_!$ induces a monad on sets which is isomorphic to $\!-\!\times\pi_1(A,\alpha)$ whenever $\alpha^*$ is faithful. If in addition the fibre functor is monadic, the category of $\pi_1(A,\alpha)$-sets is equivalent to the category of coverings over $A$.\end{prp}

\begin{proof}The second statement follows from the first because $\Cov_A\simeq PA$. For the first statement we exploit the close relationship between adjunctions fulfilling Frobenius reciprocity and group actions, cf. Townsend \cite{To} and Brugui\`eres, Lack and Virelizier \cite{BLV}. Indeed, since in a cartesian context functors are automatically comonoidal, all that is needed for a monad $T$ to be a Hopf monad, and hence to induce a group action (cf. \cite[Theorem 5.7]{BLV}), is the invertibility of the fusion operator $T(X\times TX')\to TX\times TX'$. If $T=\alpha^*\alpha_!$ this follows from Frobenius reciprocity, cf. Proposition \ref{Frobenius}a, putting $f=\alpha$ and $Y=\alpha_!X'$ and applying $\alpha^*$.

The acting group has underlying set $T(\star_\Set)$ and the category of $T$-algebras is equivalent to the category of $\!-\!\times T(\star_\Set)$-sets. Since $\el_A(\alpha_!(\star_\Set))$ yields the universal covering $\Uu_\alpha$, the group $\pi_1(A,\alpha)$ acts simply transitively on $\alpha_!(\star_\Set)$ which yields the required identification in case $\alpha^*$ is faithful.\end{proof}
Let us recall the following terminology: A covering $\xi:E\to A$ is called \emph{principal} if the action-map $\Aut(\xi)\bullet E\to E\times_A E$ is invertible, where $\Aut(\xi)$ denotes the group of automorphisms of $\xi$ fixing $A$, and $\Aut(\xi)\bullet E$ denotes a coprodut of copies of $E$ indexed by the elements of $\Aut(\xi)$. We shall say that the principal covering $\xi$ admits the \emph{Borel construction} if for any $\Aut(\xi)$-set $X$, the quotient $X\times_{\Aut(\xi)}E$ of $X\bullet E$ by the diagonal $\Aut(\xi)$-action exists.

\begin{thm}\label{descent}Let $\Ee$ be a category with discrete comprehension scheme such that
\begin{itemize}\item[(i)]All objects are locally connected;\item[(ii)]The terminal object $\star_\Ee$ is projective with respect to epicoverings;\item[(iii)]Epicoverings are strongly epimorphic inside the category of coverings;\item[(iv)]Principal coverings admit the Borel construction.\end{itemize}
Then for any connected object $A$ and any basepoint $\alpha:\star_\Ee\to A$ the fibre functor $\alpha^*$ is monadic, and the category of coverings over $A$ is equivalent to the category of $\pi_1(A,\alpha)$-sets. In particular, any two basepoints of $A$ induce isomorphic fundamental groups and isomorphic universal coverings.\end{thm}

\begin{proof}We proceed in three steps. We show (1) that $\alpha^*$ is conservative by proving that the counit of the $(\alpha_!,\alpha^*)$-adjunction is pointwise a strong epimorphism. This implies that $\alpha^*$ is faithful and hence (by Proposition \ref{monadicity}) that $\alpha^*$ factors through a conservative functor $\phi^*:\Cov_A\to\Set^{\pi_1(A,\alpha)}$. We show (2) that $\phi^*$ has a left adjoint functor $\phi_!$ and (3) that $\phi_!$ is fully faithful proving thereby that $(\phi_!,\phi^*)$ is an equivalence and $\alpha^*$ monadic. The equivalences $\Set^{\pi_1(A,\alpha)}\simeq\Cov_A\simeq\Set^{\pi_1(A,\beta)}$ imply that the fundamental groups with respect to any basepoints $\alpha,\beta$ are isomorphic. Moreover, both universal coverings $u_\alpha,u_\beta$ correspond to the regular representation of their fundamental group and are thus isomorphic in $\Cov_A$.

For (1) observe that under the equivalences $\Cov_{\star_\Ee}\simeq P(\star_\Ee)$ and $\Cov_A\simeq PA$ the fibre functor $\alpha^*$ is given by pullback along $\alpha:\star_\Ee\to A$. The universal property of $u_\alpha:U_\alpha\to A$ shows that $\Cov_A(u_\alpha,-)$ is isomorphic to $\alpha^*$, and the counit $\alpha_!\alpha^*(\xi)\to\xi$ at a covering $\xi:E\to A$ may be identified with $\Cov_A(u_\alpha,\xi)\bullet u_\alpha\to \xi$. The right cancellability of connected morphisms shows that $U_\alpha$ is connected, so that the coproduct $\Cov_A(u_\alpha,\xi)\bullet U_\alpha$ is a coproduct of connected components indexed by the elements of the fibre $\alpha^*(\xi)$. Since $A$ is connected, the restriction of $\xi:E\to A$ to any connected component of $E$ is an epicovering so that by hypothesis (ii) the fibre $\alpha^*(\xi)=\Cov_A(u_\alpha,\xi)$ contains a point in any connected component of $E$.  Computing left adjoints along $\star_\Ee\to A\to\star_\Ee$ implies then that the counit induces a surjection on connected components $\pi_0(\Cov_A(u_\alpha,\xi)\bullet U_\alpha)\to\pi_0(E)$, i.e. an epicovering. By hypothesis (iii) any such is strongly epimorphic inside the category of coverings. This shows that the counit is pointwise a strong epimorphism in $\Cov_A$.

Note that since by hypothesis (i) all objects are locally connected, distinct connected components are disjoint. Therefore, since by Corollary \ref{complement} subobjects in $\Cov_A$ are complemented, general coproducts in $\Cov_A$ are disjoint as well. Moreover, for each morphism $f:A\to B$, pullback $f^*:\Cov_A\to\Cov_B$ preserves initial objects and hence disjointness. This applies in particular to the fibre functor $\alpha^*$.

For (2) observe that the universal covering $u_\alpha$ is a principal covering. Indeed, since $\Cov_A(u_\alpha,-)$ represents $\alpha^*$ we get $\Cov_A(u_\alpha,u_\alpha)=\Aut(u_\alpha)=\pi_1(A,\alpha)$, cf. the proof of Proposition \ref{monadicity}. The counit at $u_\alpha$ is given by $\Aut(u_\alpha)\bullet U_\alpha\to U_\alpha$ which extends to the action-map $\Aut(u_\alpha)\bullet U_\alpha\to U_\alpha\times_AU_\alpha$. The latter is invertible since it induces (under the conservative fibre functor) the invertible action-map of the regular representation of $\pi_1(A,\alpha)$. It is now readily verified that for every $\pi_1(A,\alpha)$-set $X$ the Borel construction $X\times_{\pi_1(A,\alpha)}u_\alpha$ has the universal property of $\phi_!(X)$, and hence the left adjoint $\phi_!:\Set^{\pi_1(A,\alpha)}\to\Cov_A$ exists by hypothesis (iv).

For (3) it suffices to show that the unit $X\to\phi^*\phi_!(X)$ is invertible for any $\pi_1(A,\alpha)$-set $X$. But $\phi^*$ is just $\alpha^*$ equipped with its canonical $\pi_1(A,\alpha)$-action. Therefore, the unit $X\to\phi^*\phi_!(X)$ may be identified with the discrete Borel construction $X\to X\times_{\pi_1(A,\alpha)}\pi_1(A,\alpha)$ which is invertible.\end{proof}

\begin{rmk}It is surprising how little extra-conditions are needed to ensure that a discrete comprehension scheme takes values in categories of $G$-sets for discrete groups $G$. In virtue of Corollary \ref{complement} our hypothesis (iii) is precisely axiom (G3) of Grothendieck's axioms (G1)-(G6) characterising categories of sets with a continuous action by a profinite group, cf. \cite[chapter V.4]{Gr}. The four hypotheses of Theorem \ref{descent} are satisfied by the categories of (well-behaved) topological spaces, simplicial sets and groupoids equipped with the comprehension schemes of Sections \ref{topological}, \ref{simplicial} and \ref{groupoid}, but hypotheses (ii) and (iii) fail for the category of small categories with respect to both comprehension schemes $P,P'$ of Section \ref{Cat}.\end{rmk}

\begin{rmk}We end this article with a few pointers to literature where Galois-theoretical ideas are potentially related to suitable comprehension schemes.

Street and Verity \cite{SV} define comprehensive factorisation in a $2$-categorical setting and express principal coverings by means of an internal notion of \emph{torsor}.

Barr and Diaconescu \cite{BD} introduce the notion of a \emph{locally simply connected} topos. In view of Moerdijk's representation theorem for \emph{Galois toposes}  \cite{Moe}, and by analogy with the topological case, it is tempting to conjecture that there is a comprehension scheme assigning to a locally (simply) connected Grothendieck topos the Galois topos of locally constant objects therein. If this is the case then the corresponding comprehensive factorisation of a geometric morphism should be of interest.

Funk and Steinberg \cite{FS} construct a universal covering topos for each \emph{inverse semigroup} with a concrete interpretation of the associated fundamental group. We conjecture that their construction derives from a suitable comprehension scheme.

Janelidze \cite{Ja} defines Galois theory in terms of a given reflective subcategory (an axiomatisation of the full subcategory of ``discrete'' objects). He developes an abstract notion of \emph{covering extension} which subsumes the topological coverings and the central extensions in algebra \cite{JK} as special cases, and obtains a Galois-type classification for covering extensions with fixed codomain. It would be interesting to relate this axiomatic Galois theory to an existing comprehension scheme.\end{rmk}

\section*{Acknowledgements}

We thank the Max-Planck-Institute for Mathematics in Bonn and the Department of Mathematics of the University of Nice for their support and hospitality. The original seed for this work was set during several visits of RK in Nice. The trimester in 2016 on Higher Structures in Geometry and Physics at the MPIM in Bonn was an indispensable catalyst. We are especially grateful to Michael Batanin and Ross Street for inspiring discussions and to the referee for helpful comments.

RK would like to thank Yuri I. Manin for his continued support and interest. He thankfully acknowledges support from the Simons foundation under collaboration grant \# 317149.

\end{document}